\documentclass[12pt]{amsart} 

\usepackage[margin=2.9cm]{geometry}             
\usepackage{color,ulem}
\usepackage{graphicx}
\usepackage{amssymb, enumerate,bm}
\usepackage[matrix,arrow,ps]{xy}
\usepackage{epstopdf, amsmath}
\usepackage{paralist}
\newcommand{\C}{\mathbb C}

\newcommand{\R}{\mathbb R}

\newcommand{\transp}{\,^t}

\DeclareMathOperator{\spanc}{span}
\newtheorem{theo}{Theorem}[section]
\newtheorem{lemma}[theo]{Lemma}
\newtheorem{cor}[theo]{Corollary}
\newtheorem{prop}[theo]{Proposition}

\theoremstyle{remark}
\newtheorem{remark}[theo]{Remark}
\theoremstyle{example}

\theoremstyle{definition}
\newtheorem{defi}[theo]{Definition}
\numberwithin{equation}{section}

\begin{document}

\begin{abstract}
We give an explicit construction of a key  family of  stationary discs attached to a nondegenerate  model quadric in  $\C^N$ and derive a necessary condition for which (each lift) of those  stationary discs  is uniquely determined by its $1$-jet at a given point  via a local diffeomorphism. This  unique $1$-jet determination  is  a crucial step   to deduce $2$-jet determination for CR automorphisms of  generic real submanifolds in $\C^N$.

\end{abstract}

\author{Florian Bertrand  and Francine Meylan}
\title[Explicit construction of stationary discs
]{Explicit construction of stationary discs and its
consequences for nondegenerate quadrics
}

\subjclass[2010]{}

\thanks{Research of the first  author was  supported by the Center for Advanced Mathematical Sciences and by an URB grant from the American University of Beirut.}

\maketitle 




\section{Introduction}

In the papers  \cite {be-bl-me, be-me, tu3}, the authors discuss the connection between the existence of stationary  discs for generic real submanifolds in $\C^N$ and the unique $2$-jet determination of their  CR automorphisms.  
Their approach is based on a method developed by the first author and Blanc-Centi  in \cite{be-bl} which relies on the family of stationary discs introduced by Lempert \cite{le}. These invariant discs and their use in mapping problems have attracted the attention of several authors 
(see for instance \cite{hu, pa, tu, ba}). As emphasized in \cite{be-bl} (see also \cite{be,tu3}), the stationary disc method is well adapted to study jet 
determination problems, and is, to our knowledge,   the  only approach in the literature which allows the treatment of  (finitely) smooth CR automorphisms of  finitely smooth submanifolds.  Note that, under appropriate nondegeneracy conditions,  if the submanifold is {\it real-analytic} then  every $C^1$ CR automorphism is real-analytic (Theorem 2 in \cite{we}, Theorem 3.1 in \cite {BJT}). In general, the unique $2$-jet determination of CR automorphisms of real-analytic Levi nondegenerate 
hypersurfaces goes back to the works of  Cartan \cite{eca}, Tanaka \cite{ta}, and Chern and Moser  
\cite{ch-mo}. In the last twenty years, the question of jet determination for CR maps was pushed further in many important papers for real analytic submanifolds \cite{be, za, BER1, BMR, eb-la-za, la-mi, ju, ju-la, mi-za, la-mi2} and in the $ \mathcal{C}^\infty$ setting \cite{eb, eb-la, ki-za,KMZ1}. Note that the question of $2$-jet determination for real analytic submanifolds (of codimension $d>1$) is not well understood \cite{gr-me}.  

Given a model quadric $Q \subset \C^N$ and an initial stationary disc $f_0$, the key point of the stationary disc method is to construct "enough" discs, near $f_0$, attached to deformations of $Q$,  with "good" geometric properties. Such construction is usually done via an implicit 
function theorem for appropriate Banach spaces, and thus the choice of both the model quadric and the initial disc is crucial. While the construction of such discs only requires the submanifold to be strongly Levi nondegenerate (Definition \ref{defstrong}), it is necessary to impose more nondegeneracy restrictions on 
the submanifold to ensure that the family of constructed discs enjoys good geometric properties. The papers \cite {be-me, tu3} actually suggest that this problem is related to the defect of the initial disc $f_0$. In order to ensure the existence of such a nondefective disc, 
the authors rely on 
the explicit expression of - some - stationary discs attached to the given quadric. Tumanov provided in \cite{tu} their full description in case the quadric is strongly pseudoconvex with generating Levi form.  In the present paper, we provide a similar description when the model 
quadric is merely strongly Levi nondegenerate (Theorem \ref{vic0}). Using this explicit family of discs, we then study conditions that ensure the unique $1$-jet determination of their lift, and that their centers fill an open set in $\C^N$, {\it two key properties  to deduce $2$-jet determination for CR automorphisms}. As expected, such conditions are directly related to the geometry of the model quadric. More precisely, we show that if the constructed family of (lift of) discs is uniquely determined by their $1$-jet at $\zeta=1$, via a local diffeomorphism, then one of such discs must be nondefective (Theorem \ref{espoir}). Moreover, using the constructed explicit family of discs, we provide a geometric context to some of the previous nondegeneracy conditions introduced earlier in \cite{be-bl-me,be-bl} (Theorem \ref{theogen} and Theorem \ref{theogen2}). 

\section{Preliminaries}
We denote by $\Delta$ the unit disc in $\C$ and by $\partial \Delta$ its boundary.

\subsection{Strongly Levi nondegenerate  generic submanifolds}
Let $M \subset \C^{n+d}$ be a  $\mathcal{C}^{4}$ generic real submanifold of real codimension $d\ge 1$ through $p=0$ given locally by 
\begin{equation}\label{eqred0}
\begin{cases}
 r_1=\Re e  w_1- \transp\bar z A_1 z+ O(3)=0\\
\ \ \ \ \vdots \\
r_d=\Re e  w_d - \transp\bar z A_d z+ O(3)=0
\end{cases}
\end{equation}
where $A_1,\hdots,A_d$ are Hermitian matrices of size $n$. In the remainder O(3), the variables $z$ and $\Im m w$ are respectively of weight one and two. We set $r:= (r_1,...,r_d)$.
 We associate to $M$ its model quadric $Q$ given by
\begin{equation}\label{eqred1}
\begin{cases}
\rho_1= \Re e  w_1 - \transp\bar z A_1 z=0\\
\ \ \ \ \vdots \\
\rho_d=\Re e  w_d - \transp\bar z A_d z=0
\end{cases}
\end{equation}
and we write $\rho:=(\rho_1,\ldots,\rho_d)$. The following notions of nondegeneracy, particularly the second one, are due to Tumanov \cite{tu}.  
\begin{defi}\label{defstrong} 
Let $M$ be  a real submanifold given by \eqref{eqred0}. 
\begin{enumerate}[i.]
\item We say that  $M$ is {\it strongly Levi nondegenerate  at $0$} if there exists 
$b \in \Bbb R^d$ such that the matrix $\sum_{j=1}^d b_jA_j$ is invertible.  
\item We say that  $M$ is {\it strongly pseudoconvex at $0$} of if there exists 
$b \in \Bbb R^d$ such that $\sum_{j=1}^d b_jA_j$ is positive definite.
\end{enumerate}
\end{defi}

We now recall the following types of nondegeneracy  introduced in \cite{be-bl-me,be-me}.

\begin{defi}\label{definondeg}
  Let $M$ be a strongly Levi nondegenerate (at $0$) real submanifold given by \eqref{eqred0}.  Let $b\in \R^d$ be such that $A:=\sum_{j=1}^d b_jA_j$ is invertible.  
\begin{enumerate}[i.]
\item We say that  $M$ is {\it $\mathfrak{D}$-nondegenerate} at 0 if there exists $V \in \C^{n}$ such that, if $D_0$ denotes the $n \times d$ matrix whose $j^{th}$ column is
 $A_jV$, then $\Re e (\transp \overline D_0 A^{-1}D_0)$ is invertible.
\item If in addition $\transp \overline D_0 A^{-1}D_0$ is  invertible, then we say that  $M$  is {\it fully nondegenerate} at $0$.
\end{enumerate}
\end{defi}

We will need the following factorization lemma which  generalizes a classical factorization theorem due to Lempert (Theorem B p. 442 in \cite{le}) for positive definite matrix.
\begin{lemma}\label{lemfact}
   Let $M$ be a 
   strongly Levi nondegenerate (at $0$) real submanifold given by \eqref{eqred0}. Let $b\in \R^d$ be such that $A:=\sum_{j=1}^d b_jA_j$  and let $a \in \C^d$. Consider the  quadratic matrix equation 
\begin{equation}\label{confi0}
PX^2 + AX +  \transp{\overline {P}}=0
\end{equation} 
where $P:=\sum_{j=1}^d {{a}_j} A_j $ and   
 $X$ is a $n\times n$ matrix. Then for $a$ sufficiently small, we have 
 \begin{equation*}
 \sum_{j=1}^d({a_j} {\overline{\zeta}} + b_j+\overline{a_j}{\zeta})A_j=({I-\overline {\zeta}\  \transp{\overline {X}}})B({I- {\zeta}  { {X}}}), \ \ \zeta \in \partial \Delta. 
\end{equation*}
where  $X$ is the unique $n\times n$ matrix solution of \eqref{confi0} such that $\|X\|<1$  and where  $B$ is an invertible Hermitian matrix of size $n$.

 \end{lemma}
 \begin{proof} We consider $a$ small enough and $X$ the unique matrix solution of \eqref{confi0} with $\|X\|<1$. Following Tumanov \cite{tu}, we define the invertible  $n\times n$ matrix
 \begin{equation}\label{confi1}
 B:= A+ PX.
 \end{equation}
 Note that using \eqref{confi0}, we obtain directly 
 \begin{equation}\label{confi2}
 \begin{cases}
\displaystyle  BX=-\transp{\overline {P}}
 \\
\displaystyle A= B+ \transp{\overline{X} \transp{\overline{B}}}X.
 \end{cases}
 \end{equation}
 We claim that $B$ is hermitian. Indeed, due to  \eqref{confi1} and \eqref{confi2} we have 
  \begin{equation*}
  B-\transp{\overline{B}}= \transp{\overline{X}}(B-\transp{\overline{B}})X.
  \end{equation*}
  Therefore, for any positive integer $k$
  \begin{equation*}
   B-\transp{\overline{B}}= {\transp{\overline{X}}}^k(B-\transp{\overline{B}}){X}^k,
  \end{equation*}
 which implies $B=\transp{\overline{B}}$ since  $\|X\|<1$. The claim is proved and the factorization follows directly from \eqref{confi2}.
 \end{proof}
 Without loss of generality,  we assume that $\|X\|<1$ for the rest of the paper.
 Inspired by the work of Tumanov in the strongly pseudoconvex case \cite{tu}, we now introduce the following definition.

 \begin{defi}\label{defstatmin}
 Let $M$ be a 
   strongly Levi nondegenerate (at $0$) real submanifold given by \eqref{eqred0}. Let $b\in \R^d$ be such that $\sum_{j=1}^d b_jA_j$ is invertible and let  $V\in \C^n$. Consider $a \in \C^d$ sufficiently small and the  solution $X$  of \eqref{confi0} with $\|X\|<1$. We say that $M$ is {\it stationary minimal at $0$ for $(a,b,V)$} if the matrices 
 $A_1,\ldots,A_d$ restricted to the orbit space
 $$\mathcal{O}_{X, V}:={\rm span}_\R \{V, XV, {X^2}V, \ldots, {X^k}V, \ldots\}$$ are 
 $\R$-linearly independent. 
 \end{defi}
 
Note that since $X=0$ when $a=0$, $M$ is   stationary minimal at $0$ for $(0,b,V)$ if and only if the space 
${\rm span}_\R \{A_1V,\ldots,A_dV\}$ is of real dimension $d$. It follows that if  $M$ is $\mathfrak{D}$-nondegenerate then $M$ is stationary minimal for $(0,b,V)$. The converse holds in case $M$ is strongly pseudoconvex; this point was in fact  already observed by the authors in \cite{be-me}.  
It is important to point out that the above definition is independent of the choice of holomorphic coordinates. Also, if $M$ is stationary minimal at $0$ then $M$ is of finite type at $0$ with $2$ the only H\"ormander number, that is, 
 the matrices $A_1,\ldots,A_d$ are linearly independent. The converse is true in case $M$ is strongly pseudoconvex \cite{tu3}. 
 
 \begin{remark}\label{remtu}
 Tumanov  observed in \cite{tu} (see the proof of Lemma 6.7 \cite{tu}) that the spaces 
 $\mathcal{O}_{X, V}$ and  $\mathcal{O}_{X, (I-X)V}$ coincide.  In particular, the submanifold $M$ is stationary minimal at $0$ for $(a,b,V)$ if and only if it is stationary minimal at $0$ for $(a,b,(I-X)^kV)$ for any integer $k$.  
 \end{remark}
We now state  
  \begin{lemma}\label{lemstat}
   Let $M$ be a strongly Levi nondegenerate submanifold given by \eqref{eqred0}. Let $b\in \R^d$ be such that $\sum_{j=1}^d b_jA_j$ is invertible and let  $V\in \C^n$. Consider $a \in \C^d$ sufficiently small and the  solution $X$  of \eqref{confi0} with $\|X\|<1$. Then the 
  following statements are equivalent:
  \begin{enumerate}[i.]
\item The submanifold $M$ is stationary minimal at $0$ for $(a,b,V)$.
\item Assume that $\lambda_1,\ldots,\lambda_d \in \R$ are such that  $\sum_{j=1}^d\lambda_jA_jX^rV=0$ for all $r=0,1,2,\ldots$,   then $\lambda_1=\dots=\lambda_d=0$.  
\item The matrix $\Re e \left(\sum_{r=0}^{\infty} \transp \overline V{\transp \overline {X}}^rA_jA_sX^r V\right)_{j,s}$ is positive definite. 
  \end{enumerate}  
  \end{lemma}
 \begin{proof}
 The equivalence between the first two statements follows directly from Definition \ref{defstatmin}. We will only prove the equivalence between the last two statements. 
 
Suppose that $\it{ii}.$ is satisfied. Let $W=(w_1,\ldots,w_d) \in \C^d$. We have  
\begin{equation}\label{eqdefpos}
\transp \overline{W}\Re e \left(\sum_{r=0}^{\infty} \transp \overline V{\transp \overline {X}}^rA_jA_sX^r V\right)_{j,s}W = \sum_{r=0}^\infty \|D_r W\|^2 + \sum_{r=0}^\infty \|\overline{D_r} W\|^2\geq 0
\end{equation}
where $D_r$ is the $n \times d$ matrix whose $s^{th}$ column is $A_sX^r V$.  If $D_r  W= \overline{D_r} W= 0$ for all nonnegative integer $r$ then     
$$\sum_{j=1}^d \Re e (w_j)A_jX^r V=\sum_{j=1}^d \Im m (w_j)A_jX^r V=0,$$
which implies that $W=0$ by $\it{ii}.$.

Assume now that the matrix $\Re e \left(\sum_{r=0}^{\infty} \transp \overline V{\transp \overline {X}}^rA_jA_sX^r V\right)_{j,s}$ is positive definite. Let $W\in \R^d$  be such that $D_rW=0$ for all nonnegative integer $r$.
According to \eqref{eqdefpos}, it follows that $\transp \overline{W}\Re e \left(\sum_{r=0}^{\infty} \transp \overline V{\transp \overline {X}}^rA_jA_sX^r V\right)_{j,s}W$ and thus $W=0$.  
 \end{proof}
 
 The next lemma shows that the notion of stationary minimality is open.
  \begin{lemma}\label{clan}
  Let $M$ be a strongly Levi nondegenerate submanifold given by \eqref{eqred0}.  Assume that $M$ is stationary minimal at $0$ for some  $(a_0,b_0,V)$. Then $M$ is stationary minimal at $0$ for $(a,b,V)$
 for $(a,b)$ sufficiently close to $(a_0,b_0)$.   
  \end{lemma}
It is important to note that the solution matrix $X$ of \eqref{confi0} depends on $a\in \C^d$ and $b\in \R^d$ .
  \begin{proof}  
  Let  $s\leq 2n$ be such that  
  $$\mathcal{O}_{X(a_0,b_0), V}={\rm span}_\R \{V, X(a_0,b_0)V, {{X(a_0,b_0)}^2}V, \dots, {{X(a_0,b_0)}^{s}}V \}.$$
 Define the vector 
 $$\tilde{V}(a,b):=(V, X(a,b)V, {{X(a,b)}^2}V, \dots, {{X(a,b)}^s}V) \in \C^{sn}$$
and the following $sn \times sn$ matrix, $j=1\ldots,d$,
 $$
\tilde {A_j}:=\begin{pmatrix}
	A_j& & & (0) \\ &A_j & & \\ & & \ddots & \\ (0)& & &A_j
\end{pmatrix}.
$$   
 The vectors ${\tilde A_1}\tilde{V}(a_0,b_0), \dots,\tilde{A_d}\tilde{V}(a_0,b_0) $ 
 are $\R$-linearly independent.  
Hence the vectors ${\tilde A_1}\tilde{V}(a,b), \dots,\tilde{A_d}\tilde{V}(a,b)$ are  $\R$-linearly independent    for $(a,b)$ in a neighborhood of $(a_0,b_0)$
and the conclusion follows. 
\end{proof}

In the same vein, we note the following result interesting on its own.
\begin{lemma}\label{lemscal}
 Let $M$ be a strongly Levi nondegenerate submanifold given by \eqref{eqred0}.  Assume that $M$ is stationary minimal at $0$ for   $(a,b,V)$. Then $M$ is stationary minimal at $0$ for 
 $(\lambda a,\lambda b,V)$ for any $\lambda \in \R\setminus\{0\}$.   
  \end{lemma}
  
 \begin{proof}
 This is due to the homogeneity of Equation \eqref{confi0} which implies that $X(\lambda a,\lambda b)=X(a,b)$.
 \end{proof}

We end this section with the following lemma whose proof is straightforward.
\begin{lemma}\label{lemX}
 Let $Q$ be a strongly Levi nondegenerate quadric given by \eqref{eqred1} and let $b_0\in \R^d$ be such that $\sum_{j=1}^d {b_0}_jA_j$ is invertible. Let $a \in \C^d$ be small enough and let  X be the unique $n\times n$ matrix solution of \eqref{confi0} (with $b=b_0-a-\overline{a}$) with  $\|X\|<1$.  
    Then for any $s=1,\ldots,d$, we have 
\begin{equation*}
\begin{cases}
 \dfrac{\partial X}{\partial {a_s}}(0)=0\\
 \\
 \displaystyle \dfrac{\partial X}{\partial  \overline a_s}(0)=\dfrac{\partial X}{\partial {\Re e a_s}}(0)=-\left(\sum_{k=1}^d {b_0}_kA_k \right)^{-1}A_s.
\end{cases}
\end{equation*}
\end{lemma}
In what follows, we denote by $X_{\Re e a_s}$ the derivative $\dfrac{\partial X}{\partial {\Re e a_s}}$, $s=1,\ldots,d$.

\subsection{Stationary discs}

 Let $M$ be a $\mathcal{C}^{4}$ generic  real submanifold of $\C^N$ of codimension $d$ given by  (\ref{eqred0}). A holomorphic disc $f: \Delta \to \C^N$ continuous up to  $\partial \Delta$ is  {\it attached to a $M$} if $f(\partial \Delta) \subset M.$  
 The following definition is due to Lempert \cite{le} for hypersurfaces and to Tumanov \cite{tu} for  higher codimension submanifolds.
 \begin{defi}
A holomorphic disc $f: \Delta \to \C^N$ continuous up to  $\partial \Delta$ and attached to  $M$ is {\it stationary for $M$} if there 
exists a  holomorphic lift $\bm{f}=(f,\tilde{f})$ of $f$ to the cotangent bundle $T^*\C^{N}$, continuous up to 
 $\partial \Delta$ and such that for all $\zeta \in \partial\Delta,\ \bm{f}(\zeta)\in\mathcal{N}M(\zeta)$
where
\begin{equation*}
\mathcal{N}M(\zeta):=\{(z,w,\tilde{z},\tilde{w}) \in T^*\C^{N} \ | \ (z,w) \in M, (\tilde{z},\tilde{w}) \in 
\zeta N^*_{(z,w)} M\setminus \{0\} \},
\end{equation*}
and where 
$$N^*_{(z,w)} M=\spanc_{\R}\{\partial r_1(z,w), \ldots, \partial r_d(z,w)\}$$ is the conormal fiber at $(z,w)$ of $M$. 
The set of these lifts $\bm{f}=(f,\tilde{f})$, with $f$ nonconstant, is denoted by $\mathcal{S}(M)$.
\end{defi}
We note that a disc $f \in \mathcal{S}(M)$ if there exist $d$ real valued functions $c_1, \ldots, c_d : \partial \Delta \to \R$ such that $\sum_{j=1}^dc_j(\zeta)\partial r_j(0)\neq 0$ for all $\zeta \in \partial \Delta$   and such that the map 
\begin{equation*}
\zeta \mapsto \zeta \sum_{j=1}^dc_j(\zeta)\partial r_j\left(f(\zeta), \overline{f(\zeta)}\right)
\end{equation*}
defined on $\partial \Delta$ extends holomorphically on $\Delta$.

\section{Explicit construction of stationary discs for quadric submanifolds}

\subsection{Explicit construction of stationary discs}
Let $Q\subset \C^N$ be  a quadric submanifold of real codimension $d$  given by  (\ref{eqred1}). 
In the  recent papers \cite{be-bl-me} with Blanc-Centi, and \cite{be-me}, we worked with a special family of lifts $\bm {f}=(h,g,\tilde{h},\tilde{g}) \in \mathcal{S}(Q)$ of the form
 \begin{equation}\label{eqinit}
 \bm {f}=\left((1-\zeta)V,2(1-\zeta)\transp {\overline{V}}A_1V,\ldots,2(1-\zeta)\transp \overline{V}A_dV,(1-\zeta) \transp { \overline{V}}\sum_{j=1}^b {b}_jA_j, \frac{\zeta}{2} b\right),
 \end{equation} 
 where $V\in \C^n$ and   $b\in \R^d$ is such that $\sum_{j=1}^d b_jA_j$ is invertible. This special family of lift can been used  
 to obtain unique jet determination properties for $\mathfrak{D}$-nondegenerate submanifolds. Nevertheless, this class of submanifolds is the largest one can treat by working with discs of the form \eqref{eqinit}.  
 Therefore, in order to study jet determination problems for  larger classes of strongly Levi nondegenerate submanifolds, it is crucial to work with more (explicit) stationary discs. This is precisely the purpose of Theorem  \ref{vic0} in which we describe explicitly stationary discs attached to $Q$. 

Before stating the main theorem, we need to introduce the following. Let $b\in \R^d, a \in \C^d$, $P$ and $X$ be as  Lemma \ref{lemfact}. Denote by  $\mathcal{M}_n(\C)$ 
the space of square matrices of size $n$ with complex coefficients. Consider the  linear map $\psi: \mathcal{M}_n(\C) \to \mathcal{M}_n(\C)$ defined by 
\begin{equation}\label{eqPsi}
\psi(M)=M- {\transp \overline {X}} M X.\\
\end{equation}
Due to the fact that $a$ is small and $\|X\|<1$, the map $\psi$ is invertible with inverse
$$\psi^{-1}(M)=\sum_{r=0}^\infty {\transp \overline {X}}^r M X^r.$$
Note that $M$ is Hermitian if and only if $\psi(M)$ is Hermitian.   

In Theorem \ref{vic0}, we focus on lifts of discs attached to a fixed point in the cotangent bundle. We fix $b_0 \in  \Bbb R^d$  such that $\sum_{j=1}^d {b_0}_jA_j$ is invertible and
 we define $\mathcal{S}_0(Q) \subset \mathcal{S}(Q)$ to be the subset of lifts whose value at $\zeta=1$ is $(0,0,0,b_0/2)$.  
 Consider an initial  disc $\bm {f_0} \in \mathcal{S}_0(Q)$ given by
 \begin{equation*}
 \bm {f_0}=\left((1-\zeta)V_0,2(1-\zeta)^t \overline{V_0}A_1V_0,\ldots,2(1-\zeta)^t \overline{V}A_dV_0,(1-\zeta) ^t \overline{V_0}(\sum {b_0}_jA_j), \frac{\zeta}{2} b_0\right),
 \end{equation*} 
with $V_0 \in\C^n$. We then obtain the following explicit expression for lifts of stationary discs near $\bm {f_0}$.

\begin{theo}\label{vic0}
Let $Q$ be a strongly Levi nondegenerate quadric given by \eqref{eqred1}. Then stationary discs $f=(h,g)$ with lifts $\bm{f}=(h,g,\tilde{h},\tilde{g}) \in \mathcal{S}_0(Q)$   near  $\bm {f_0}$ are exactly of the form
\begin{equation}\label{eqform}
\begin{cases}
h(\zeta)= V-\zeta(I-\zeta X)^{-1}(I-X)V \\
\\
g_j(\zeta)={^t\overline{V}}A_j V-2^t \overline {V}A_j \zeta (I-\zeta X)^{-1}(I-X)V + \\
\hspace{1.4cm}^t \overline {V}(I-^t \overline {X})K_j(I + 2\zeta X(I-\zeta X)^{-1})(I-X) V +^t \overline {V}(^t\overline {X}K_j-K_jX))V\\
\end{cases}
\end{equation}
where  $V \in \Bbb C^n$ (close to $V_0$), $a \in \C^d$ is sufficiently small, $X$ is the unique $n\times n$ matrix solution of \eqref{confi0} (with $b=b_0- a-\overline{a}$) with $\|X\|<1$, and $K_j=\psi^{-1}(A_j)$ is Hermitian,  j=1, \dots,d.
\end{theo}

\begin{proof}
Let $\bm{f}=(h,g,\tilde{h},\tilde{g}) \in \mathcal{S}_0(Q)$ be a lift of stationary disc. Consider $d$ real valued functions 
$c_1, \ldots, c_d : \partial \Delta \to \R$ such that $\sum_{j=1}^dc_j(\zeta)\partial \rho_j(0)\neq 0$ for all $\zeta \in \partial \Delta$ and such that the map 
$\zeta \mapsto \zeta \sum_{j=1}^dc_j(\zeta)\partial \rho_j(f(\zeta), \overline{f(\zeta)})$ defined on $\partial \Delta$ extends holomorphically on 
$\Delta$. It follows in particular that  each function $c_j$ is of the form 
\begin{equation*}
c_j(\zeta)= {a_j}\overline{\zeta}+b_j+\overline {a_j}\zeta,
 \end{equation*} 
 where $a_j \in \C$ and $\ b_j \in \R$. We set $a=(a_1,\ldots,a_d)$. So the lift components of  $\bm{f}$ are of the form   
 \begin{equation}\label{vac1}
\tilde{h}(\zeta)=-\zeta\transp\overline{h(\zeta)} \left({\sum_{j=1}^d ({a_j}\overline{\zeta}+b_j+\overline {a_j}\zeta)A_j}\right)
\end{equation}
and 
\begin{equation}\label{vac2}
\tilde{g}(\zeta)=\frac{a+b\zeta+\overline{a}\zeta^2}{2}
\end{equation}
with $b=b_0- a-\overline{a} \in \R^d$ to ensure that $\tilde{g}(1)=b_0/2$.

Consider now $a \in \C^d$ small enough and the corresponding solution $X$  of \eqref{confi0} with $\|X\|<1$.
Using  \eqref{vac1} and Lemma  \ref{lemfact}, we obtain,   by definition of the stationarity, that  the map
\begin{equation*}
\zeta \mapsto \zeta\transp\overline{h(\zeta)} ({I-\overline {\zeta}\  \transp{\overline {X}}})B({I- {\zeta}  { {X}}})
\end{equation*}
defined on $\partial \Delta$ extends holomorphically to the unit disc. Therefore the map
$$\zeta \mapsto\zeta({I- \overline{\zeta}\  \overline{X} })\overline{h(\zeta)}$$ extends holomorphically to the unit disc. Writing $h(\zeta)=\sum_{j=0}^\infty \alpha_j\zeta^j$, this implies directly that  $\alpha_j=X^{j-1}\alpha_1$ for  $j\geq 2$ and so 
the component $h$ is precisely of the form 
\begin{equation*}
h(\zeta)= h(0) + \zeta ({I- {\zeta}\  X })^{-1}h'(0).
\end{equation*}
Since $h(1)=0$ we obtain $h'(0)=-(I-X)h(0),$ and so the first part of \eqref{eqform} follows with $V=h(0)$.
The form of the component $g$ is obtained by using the fact that the disc is attached to the quadric Q and thus satisfies $\Re e g_j=\transp\overline{h}A_jh$ for $j=1,\ldots,d$. The computation is straightforward and leads to      
 \begin{eqnarray*}
 g_j(\zeta)&=& ^t \overline {V}A_j V-2^t \overline {V}A_j \zeta (I-\zeta X)^{-1}(I-X)V +\\
& &^t \overline {V}(I-{^t \overline {X}})K_j(I + 2\zeta X(I-\zeta X)^{-1})(I-X) V +iy_j  \\
\end{eqnarray*}
 with $y_j \in \R$ and $K_j=\psi^{-1}(A_j)$, where $\Psi$ is defined in \eqref{eqPsi}.
Finally,  since $g(1)=0$ and using the fact that $\psi(K_j)=A_j$, we obtain 
 $$iy_j={^t \overline {V}}(^t\overline {X}K_j-K_jX))V.$$ 
 This achieves the proof of the theorem.
\end{proof}

\begin{remark}\label{vic5}
The above theorem shows that, for a strongly Levi nondegenerate quadric given by \eqref{eqred1}, $\mathcal{S}_0(Q)$ is parametrized by $a \in \C^d$ and $V\in \C^n$ near $\bm {f_0}$, that is, $2n+2d=2N$ real parameters. That result was obtain via an implicit function theorem in \cite{be-bl-me} and explicitly in the case of strongly pseudoconvex quadric with generating Levi form form in \cite{tu}. Also, note that in case $a=(0,\ldots,0) \in \C^d$, one recovers the special family of lift given by \eqref{eqinit}. 
\end{remark}
In the above theorem, although it is important that the parameter $a \in \C^d$ is sufficiently small, no condition is need for the parameter $V \in \C^n$. We only require $V$ close to a given $V_0$ to make sure that the constructed family of lifts is in a
 neighborhood of the initial disc $\bm{f_0}$.

 \subsection{Nondefective stationary discs}

In what follows, we discuss the notion of defect of a stationary disc and its relation with Definition \ref{defstatmin}. Following \cite{ba-ro-tr}, a stationary disc $f$ is {\it defective} if it admits a lift 
$\bm{f}=(f,\tilde{f}): \Delta \to T^*\C^N$ such that  $\displaystyle 1/\zeta.\bm{f}=(f,\tilde{f}/\zeta)$ is holomorphic on $\Delta$. The discs is {\it nondefective} in case it is  not defective. 
For a quadric $Q\subset \C^N$ of the form \eqref{eqred1}, in view of \eqref{vac1} and \eqref{vac2}, a stationary disc $f=(h,g)$ for $Q$ is defective if 
there exists $c=(c_1,\ldots,c_d) \in \R^d\setminus\{0\}$ such that the map
 $$\zeta \mapsto c\partial_z \rho(f(\zeta))= \sum_{j=1}^d c_j \partial_z \rho_j(f(\zeta))=-\transp\overline{h(\zeta)}{\sum_{j=1}^d c_jA_j}$$ defined on $\partial \Delta$ extends holomorphically on $\Delta$. 
In \cite{be-me}, we observed that a stationary disc of the form
$$f=\left((1-\zeta)V,2(1-\zeta)\transp \overline{V}A_1V,\ldots,2(1-\zeta)\transp \overline{V}A_dV\right),$$ where $V\in \C^n$ and with lift of the special form \eqref{eqinit}, 
 is nondefective if and only if $Q$ is stationary minimal at $0$ for $(0,b,V)$ (see Lemma 3.3 in \cite{be-me}). In general, we have   

\begin{prop}\label{propstat} 
Let $Q$ be a strongly Levi nondegenerate quadric given by \eqref{eqred1} and  let $b_0 \in \R^d$ be such that $\sum_{j=1}^d {b_0}_jA_j$ is invertible. 
Consider a  stationary disc $f$ of the form \eqref{eqform} with lift $\in \mathcal{S}_0(Q)$.   
The following statements are equivalent:
\begin{enumerate}[i.]
\item  The disc $f$ is nondefective.  
\item The quadric $Q$ is stationary minimal at $0$ for $(a,b_0-a-\overline{a},h'(0))$.  
\item The quadric $Q$ is stationary minimal at $0$ for $(a,b_0-a-\overline{a},h'(1))$.
\item The quadric $Q$ is stationary minimal at $0$ for $(a,b_0-a-\overline{a},h(0))$.
\end{enumerate}
\end{prop}
\begin{proof}
The equivalence of the statements {\it ii.}, {\it iii.} and {\it iv.} follows from  Remark \ref{remtu} and the fact that $h'(0)=-(I-X)V$ and  $h'(1)=-(I-X)^{-1}V$and $h(0)=V$.  
We then prove that {\it  i.} implies {\it  iv.}. Assume that there exist $(\lambda_1,\ldots,\lambda_d) \in \R^d\setminus\{(0,0,\ldots,0)\}$ such that  $\sum_{j=1}^d\lambda_jA_jX^rV=0$ for all $r=0,1,2,\ldots$. We claim that the disc $f$ is defective since it admits it admits  a lift 
 $\bm{f}=(h,g,0,\zeta \lambda_1/2,\ldots,\zeta \lambda_d/2)$ such that $\displaystyle 1/\zeta.\bm{f}$ is holomorphic on $\Delta$. Indeed, we have 
 
 \begin{eqnarray*}
  \transp\overline{h(\zeta)}{\sum_{j=1}^d \lambda_jA_j} &=&   \underbrace{{\transp\overline{V}}\sum_{j=1}^d \lambda_jA_j}_{=0}-\overline{\zeta} {\transp\overline{V}}(I-\overline{\zeta} \transp\overline{X})^{-1}(I-\transp\overline{X})\sum_{j=1}^d \lambda_jA_j \\
  \\
    &=& -\overline{\zeta} \sum_{r=0}^\infty \overline{\zeta}^r\underbrace{\transp\overline{V}\transp\overline{X^r}\sum_{j=1}^d \lambda_jA_j}_{=0}+\overline{\zeta} \sum_{r=0}^\infty \overline{\zeta}^r\underbrace{\transp\overline{V}\transp\overline{X^{r+1}}\sum_{j=1}^d \lambda_jA_j}_{=0}=0.\\
  \end{eqnarray*}
We now prove that {\it ii.} implies {\it i.}. Assume that $f$ is defective. There exist $\lambda_1,\ldots,\lambda_d \in \R$ such that $\transp\overline{h(\zeta)}{\sum_{j=1}^d \lambda_jA_j}$ extends holomorphically on $\Delta$. Set $\tilde{V}=(I-X)V$. Since  \begin{eqnarray*}
 \transp\overline{h(\zeta)}{\sum_{j=1}^d \lambda_jA_j} &=& \transp\overline{V}\sum_{j=1}^d \lambda_jA_j- \sum_{r=0}^\infty \overline{\zeta}^{r+1}{\transp(\overline{X^{r}\tilde{V}})}\sum_{j=1}^d \lambda_jA_j \\
   \end{eqnarray*}
we have  ${\transp(\overline{X^{r}\tilde{V}})}\sum_{j=1}^d \lambda_jA_j=0$ for all $r=0,1,2,\ldots$ which shows that  $Q$ is not stationary minimal at $0$ for $(a,b_0-a-\overline{a},\tilde{V}).$ 
\end{proof}

\section{1-jet determination of stationary discs}
Let $Q$ be a strongly Levi nondegenerate quadric given by \eqref{eqred1}. 
 Consider the $1$-jet map $$\mathfrak j_{1}:\bm{f} \mapsto (\bm{f}(1),\bm{f}'(1))$$ at $\zeta=1$. We focus on lifts $\bm{f} \in \mathcal{S}_0(Q)$. Since $\bm{f}(1)=(0,0,0,b_0/2)$ where $b_0 \in \R^d$ is fixed, we identify the  $1$-jet map with the derivative map  
 $\bm{f} \mapsto \bm{f}'(1)$ at $\zeta=1$. This map may be expressed explicitly in view of  Theorem \ref{vic0} and Remark \ref{vic5}:
 \begin{prop}\label{propder}
In the context of Theorem \ref{vic0}, the  $1$-jet map 
 $$\mathfrak j_{1}: \C^d \times \C^n \to  \C^n \times \R^d \times \C^d$$ at $\zeta=1$ is given by 
\begin{equation*}
(a,V)\mapsto \bm{f} \mapsto  (h'(1), g'(1), \tilde {g}'(1))=
\left(-(I-X)^{-1}V,-2 \transp\overline{V}K_1V, \ldots, -2\transp\overline{V}K_dV,\dfrac{1}{2} ( b_0 -2i\Im m {a})\right).  
\end{equation*}
\end{prop}
Note that the component 
$$\tilde {h}'(1)=\transp\overline{h'(1)} \left(\sum_{j=1}^d (a_j+b_j+\overline{a_j})A_j\right)=\transp\overline{h'(1)} \left(\sum_{j=1}^d b_{0j}A_j\right)$$ is omitted since the information it carries is redundant due to the invertibility of the matrix
$\sum_{j=1}^d b_{0j}A_j$.
\begin{proof}[Proof of Proposition \ref{propder}]
We first note that the $1$-jet map of $\zeta(I-\zeta X)^{-1}$ at $\zeta=1$ is given by $(I-X)^{-2}$.  It follows directly that using the form of $h$
given by \eqref{eqform}, we obtain 
$$h'(1)= -(I-X)^{-1}V$$
and
$$g_j'(1)=2\transp\overline{V} (-A_j +K_jX -\transp\overline XK_j X)(I-X)^{-1} V=-2\transp\overline{V}K_jV$$
for $j=1,\ldots,d$. We also have, using the expression \eqref{vac2},
$$\tilde{g}'(1)=\dfrac{1}{2}(b_0+\overline a- a) .$$ This achieves the proof of the lemma.
\end{proof}
After performing changes of variables in both the source and the target spaces,
 the $1$-jet map $\mathfrak j_{1}$ from the previous corollary may be  written as 
\begin{equation}\label{riki1} 
\mathfrak j_{1}: (a,V) \mapsto 
(V, \transp\overline{V}(I-\transp\overline {X})K_1(I-X)V,\ldots, \transp\overline{V}(I-\transp\overline {X})K_d(I-X)V, \Im m a).
\end{equation} 
Due to the form of the differential map of $\mathfrak j_{1}$ at $(a,V) \in \C^d\times \C^n$ we obtain directly
\begin{cor}\label{cordif}
The $1$-jet map $\mathfrak j_{1}$ is a local diffeomorphism at $(a,V) \in \C^d\times \C^n$ if and only if 
the following $d\times d$ matrix 
$$\left(\dfrac{\partial}{\partial {\Re e a_s}}\transp\overline{V}(I-\transp\overline {X})K_j(I-X)V\right)_{j,s}$$ is invertible.
\end{cor}
In what follows, we investigate the invertibility of that matrix.
We denote  by $A_H$ the Hermitian part of a square matrix $A,$ namely 
$$A_H:=\dfrac{1}{2}(A+\transp\overline{A}).$$ Note that $\psi^{-1}(A_H)=(\psi^{-1}(A))_H$, where $\psi$ is defined by \eqref{eqPsi}.  
We need the following lemma.

\begin{lemma}\label{tortue}  Let $Q$ be a strongly Levi nondegenerate quadric given by \eqref{eqred1} and let $b_0 \in \R^d$ be such that $\sum_{j=1}^d {b_0}_jA_j$ is invertible.
 Let $a \in \C^d$ be small enough and let  X be the unique $n\times n$ matrix solution of \eqref{confi0} (with $b=b_0-a-\overline{a}$) such that $\|X\|<1$.
 Then for any $s=1,\ldots,d$, we have 
%

\begin{equation}\label{eqder}
\dfrac{\partial}{\partial {\Re e a_s}}(I-\transp\overline {X})K_j(I-X) =  -2\left(\left(I-\transp \overline {X}\right)^2\psi^{-1}(K_jX_{\Re e a_s})\right)_H.\\
\end{equation}

The term $\left(I-\transp \overline {X}\right)^2 \psi^{-1}(K_jX_{\Re e a_s})$ depends on $a$, and in particular  we have, for $a=0$,
$$\left(I-\transp \overline {X(0)}\right)^2 K_j(0)X_{\Re e a_s}(0)=-A_j\left(\sum_{k=1}^d {b_0}_kA_k \right)^{-1}A_s.$$ 

\end{lemma}

\begin{proof}
Since $\psi(K_j)=A_j$ , we have  
\begin{equation*}
\dfrac{\partial}{\partial {\Re e a_s}}(I-\transp\overline {X})K_j(I-X)=
\dfrac{\partial}{\partial {\Re e a_s}}(2K_j-K_jX -\transp\overline {X}K_j)=2\dfrac{\partial}{\partial {\Re e a_s}}(K_j-K_jX)_H
\end{equation*}
and
\begin{equation*}
\psi\left(\dfrac{\partial}{\partial {\Re e a_s}}K_j\right)=2\left({\transp \overline {X}}K_jX_{\Re e a_s}\right)_H.
\end{equation*}
Inverting $\psi$ leads to  
\begin{eqnarray*}
 \dfrac{\partial}{\partial {\Re e a_s}}K_j&=&2\psi^{-1}\left({\transp \overline {X}}K_jX_{\Re e a_s}\right)_H=2\left(\transp\overline {X}\underbrace{\psi^{-1}(K_j X_{\Re e a_s})}_{=:B}\right)_H.
 \end{eqnarray*}
We also have
\begin{equation*}
  \dfrac{\partial}{\partial {\Re e a_s}}K_jX= \left(\dfrac{\partial}{\partial {\Re e a_s}}K_j\right)X+K_jX_{\Re e a_s}.
\end{equation*}
It follows that
\begin{eqnarray*}
\dfrac{\partial}{\partial {\Re e a_s}}(I-\transp\overline {X})K_j(I-X)&=&2\dfrac{\partial}{\partial {\Re e a_s}}(K_j-K_jX)_H\\
\\
&=&2\left(2(\transp\overline {X}B)_H-2(\transp\overline {X}B)_HX-K_jX_{\Re e a_s}\right)_H\\
\\
&=&2\left(2\transp\overline {X}B-\transp\overline {B}X^2\underbrace{-\transp\overline {X}BX-K_jX_{\Re e a_s}}_{-B}\right)_H\\
\\
&=&2\left(2\transp\overline {X}B-\transp\overline {X^2}B-B\right)_H\\
\\
&=&-2\left((I-\transp\overline {X})^2B\right)_H.\\
\end{eqnarray*}
This concludes the proof of \eqref{eqder}. The proof of the second statement of Lemma  \ref{tortue} follows directly from Lemma \ref{lemX} and the fact that $\psi$ is the identity when $a=0.$  
\end{proof}

As a direct consequence, we obtain 
\begin{theo}\label{theogen}
Let $Q$ be a strongly Levi nondegenerate quadric given by \eqref{eqred1} and let $b_0 \in \R^d$ be such that $\sum_{j=1}^d {b_0}_jA_j$ is invertible. Then the $1$-jet map $\mathfrak j_{1}$ \eqref{riki1} is a local diffeomorphism at $(0,V)$ if and only if the $d \times d$ matrix 
\begin{equation*}
\Re e \left(\transp\overline{V}A_j\left(\sum_{k=1}^d {b_0}_kA_k \right)^{-1}A_sV\right)_{j,s}
\end{equation*} 
is invertible. In other words, the $1$-jet map $\mathfrak j_{1}$ is a local diffeomorphism at $(0,V)$ if and only if the quadric $Q$ is 
$\mathfrak{D}$-nondegenerate (with $V$). 
\end{theo}
\begin{proof}
According to Lemma \ref{tortue}, we have for any $s=1,\ldots,d$,
  \begin{equation*} 
 \dfrac{\partial}{\partial {\Re e a_s}}\left((I-\transp\overline {X})K_j(I-X)\right)(0)= 2  \left(A_j\left(\sum_{k=1}^d {b_0}_kA_k\right)^{-1}A_s\right)_H.
 \end{equation*}
Now note that for any $V \in \C^n$ and any $n\times n$ matrix, we have 
\begin{equation}\label{espoir0}
\transp \overline{V} A_H V=\Re e(\transp\overline{V} A V).
\end{equation} Thus  
\begin{equation*} 
 \dfrac{\partial}{\partial {\Re e a_s}}\transp \overline{V} (I-\transp\overline {X})K_j(I-X)V= 2 \Re e \left(\transp \overline{V}  A_j\left(\sum_{k=1}^d {b_0}_kA_k\right)^{-1}A_s V\right).
 \end{equation*}
\end{proof}
We want to emphasize that the $\mathfrak{D}$-nondegeneracy of $Q$ (see Definition \ref{definondeg}) is not a purely technical condition. In fact, it is  important to note that  
Theorem \ref{theogen} shows the geometric and adapted nature of this nondegeneracy condition. 
In general, it is important to find necessary and sufficient conditions (more trackable and geometric than the one given in Corollary  \ref{cordif}) to ensure that the  $1$-jet map $\mathfrak j_{1}$ is a local diffeomorphism. In the next theorem, we show that the stationary minimality of the quadric $Q$  is necessary. In a forthcoming paper, we will address and study the question of the sufficient condition. 
\begin{theo}\label{espoir}
Let $Q$ be a strongly Levi nondegenerate quadric given by \eqref{eqred1} and  let $b_0 \in \R^d$ be such that $\sum_{j=1}^d {b_0}_jA_j$ is invertible. Assume that  the $1$-jet map $\mathfrak j_{1}$ \eqref{riki1} is a local diffeomorphism at some $(a_0,V)$, for $a_0$ sufficiently small.   Then $Q$ is stationary minimal at $0$ for $(a_0,b_0-a_0-\overline{a_0},V)$.   
\end{theo}

\begin{proof}
Assume that  the $1$-jet map $\mathfrak j_{1}$ is a local diffeomorphism at  $(a_0,V)$ and suppose by contradiction that $Q$ is not stationary minimal at $0$ for $(a_0,b_0-a_0-\overline{a_0},V)$. According to Remark \ref{remtu}, $Q$ is not stationary minimal at $0$ for $(a_0,b_0-a_0-\overline{a_0},V')$ where $V'=(I-X)^2V$, where  $X$  is the unique $n\times n$ matrix solution of \eqref{confi0} (with $b=b_0-a_0-\overline{a_0}$) such that $\|X\|<1$.
In particular, there exists  $W\in \Bbb R^d \setminus{\{0\}}$ such that for all integer $r \geq 0$, $D_rW=0$, where $D_r$ is the $n \times d$ matrix whose $j^{th}$ column is $A_jX^r V'$.
According to Lemma \ref{tortue} and \eqref{espoir0}, we may rewrite
$-\displaystyle \frac{1}{2}\dfrac{\partial}{\partial {\Re e a_s}}\transp \overline{V} (I-\transp\overline {X})K_j(I-X)V$ as follows
\begin{eqnarray*}
\transp \overline{V}\left(\sum_{r=0}^{\infty} {\transp \overline {X}}^r\left(I-\transp \overline {X}\right)^2 K_jX_{\Re e a_s}X^r\right)_HV 
&= &\Re e\left(\sum_{r=0}^{\infty} \transp\overline{V'}{\transp \overline {X}}^rK_jX_{\Re e a_s}X^r V\right)\\
\\
&= &\sum_{r=0}^{\infty}\sum_{l=0}^\infty\Re e\left(\transp\overline{V'}{\transp \overline {X}}^{r+\ell}A_jX^{\ell}X_{\Re e a_s}X^r V\right)\\
\end{eqnarray*}
It follows that $\transp W\left (\dfrac{\partial}{\partial {\Re e a_s}}\transp \overline{V} (I-\transp\overline {X})K_j(I-X)V  \right )_{j,s}$ only involves terms of the form $\transp W\transp\overline{D_{r+\ell}}$ or $\transp W\transp D_{r+\ell}$ and is thus equal to zero. According to Corollary \ref{cordif}, this is a contradiction.
\end{proof}

\section{Filling properties of stationary discs}
Let $Q$ be a strongly Levi nondegenerate quadric given by \eqref{eqred1}. 
In this section, we consider  the center evaluation map 
$$\Psi: \bm{f} \mapsto f(0)=(h(0), g(0)),$$
where $\bm{f} \in \mathcal{S}_0(Q)$.  We obtain immediately  from Theorem \ref{vic0} the following explicit expression of $\Psi$: 
\begin{prop}
In the context of Theorem \ref{vic0}, the center evaluation map 
$$\Psi:  \C^d \times \C^n \to  \C^n \times \C^d $$ at $\zeta=0$ is given by 
\begin{equation*} 
(a,V)  \mapsto \bm{f} \mapsto f(0)=(V, \transp\overline{V}2K_j(I-X)V).
\end{equation*} 
\end{prop}
According to that explicit form, we have
\begin{cor}
The center evaluation map $\Psi$  is a local diffeomorphism at $(a,V) \in \C^d\times \C^n$ if and only if 
the following $d\times d$ matrix 
$$\left(\dfrac{\partial}{\partial {\Re e a_s}}\transp\overline{V}K_j(I-X)V\right)_{j,s}$$ is invertible.
\end{cor}

In the next theorem, we investigate the invertibility of that matrix in the case of a strongly Levi nondegenerate quadric for $a=0$ and any $V \in \C^n$. We have
\begin{theo}\label{theogen2}
Let $Q$ be a strongly Levi nondegenerate quadric given by \eqref{eqred1} and let $b_0 \in \R^d$ be such that $\sum_{j=1}^d {b_0}_jA_j$ is invertible. Then the center evaluation map  $\Psi$  is a local diffeomorphism at $(0,V)$ if and only if the $d \times d$ matrix 
\begin{equation*}
\left(\transp\overline{V}\transp\overline{A_j}\left(\sum_{k=1}^d {b_0}_kA_k \right)^{-1}A_sV\right)_{j,s}
\end{equation*} 
is invertible.
\end{theo}

It is remarkable that in Theorem \ref{theogen2}, the condition under which the center evaluation map is a local diffeomorphism at $(0,V)$ is precisely the  invertibility condition in ii. of Definition \ref{definondeg} of  full nondegeneracy. This illustrates the relevance of this notion of nondegeneracy and its relation with the geometric properties of stationary discs with lift of the form \eqref{eqinit}. 

\begin{proof}[Proof of Theorem \ref{theogen2}]
Using Lemma \ref{lemX} and the proof of Lemma \ref{tortue}, we obtain for any $j,s=1,\ldots,d$,
 \begin{eqnarray*}
  \dfrac{\partial}{\partial {\Re e a_s}}\left(K_j(I-X)\right)(0)&=&  \dfrac{\partial K_j}{\partial {\Re e a_s}}(0)-A_j\dfrac{\partial X}{\partial {\Re e a_s}}(0)\\
 \\
 & =& \underbrace{2\psi^{-1}\left({\transp \overline {X}}K_jX_{\Re e a_s}\right)_H(0)}_{=0}+A_j\left(\sum_{k=1}^d {b_0}_kA_k \right)^{-1}A_s \\
\end{eqnarray*}
and the proof follows.
\end{proof}

\vskip 1cm
{\small
\noindent Florian Bertrand\\
Department of Mathematics,\\
American University of Beirut, Beirut, Lebanon\\
{\sl E-mail address}: fb31@aub.edu.lb\\

\noindent Francine Meylan \\
Department of Mathematics\\
University of Fribourg, CH 1700 Perolles, Fribourg\\
{\sl E-mail address}: francine.meylan@unifr.ch\\
} 

\end{document}